\documentclass[12pt,reqno]{article}

\usepackage[usenames]{color}
\usepackage{amssymb}
\usepackage{amsmath}
\usepackage{amsthm}
\usepackage{amsfonts}
\usepackage{amscd}
\usepackage{graphicx}

\usepackage[colorlinks=true,
linkcolor=webgreen,
filecolor=webbrown,
citecolor=webgreen]{hyperref}

\definecolor{webgreen}{rgb}{0,.5,0}
\definecolor{webbrown}{rgb}{.6,0,0}

\usepackage{color}
\usepackage{fullpage}
\usepackage{float}

\usepackage{graphics}
\usepackage{latexsym}

\begin{document}
	
	\theoremstyle{plain}
	\newtheorem{theorem}{Theorem}
	\newtheorem{corollary}[theorem]{Corollary}
	\newtheorem{lemma}{Lemma}
	\newtheorem{example}{Examples}
	\newtheorem*{remark}{Remark}
	\newtheorem{prop}{Proposition}
	
	\begin{center}
		\vskip 1cm
		{\LARGE\bf
			Fibonacci-harmonic sums}
		
		\vskip 1cm
		
		{\large
			Kunle Adegoke \\
			Department of Physics and Engineering Physics \\ Obafemi Awolowo University, 220005 Ile-Ife \\ Nigeria \\
			\href{mailto:adegoke00@gmail.com}{\tt adegoke00@gmail.com}
			
			\vskip 0.2 in
			
		Segun Olofin Akerele \\
			Department of Mathematics \\ University of Ibadan, 200211 Ibadan \\ Nigeria \\
			\href{mailto:akereleolofin@gmail.com}{\tt akereleolofin@gmail.com}
			
			\vskip 0.2 in
			
			Robert Frontczak \\
			Independent Researcher, 72764 Reutlingen \\ Germany \\
			\href{mailto:robert.frontczak@web.de}{\tt robert.frontczak@web.de}
		}
		
	\end{center}
	
	\vskip .2 in
	
	\begin{abstract}
		We offer several new summation identities involving harmonic numbers, odd harmonic numbers, and Fibonacci numbers.
		Our results are derived using three different approaches: partial summation, polynomial identities and binomial transformation.
	\end{abstract}
	
	
	
	\bigskip

	\section{Introduction}
	
	The Fibonacci numbers $F_j$ and the Lucas numbers $L_j$ are defined, for \text{$j\in\mathbb Z$}, through the recurrence relations
	\begin{align*}
		F_j = F_{j-1}+F_{j-2},& \quad j\geq 2,\quad F_0=0,\, F_1=1,\\
		L_j = L_{j-1}+L_{j-2},& \quad  j\geq 2, \quad L_0=2, \, L_1=1,
	\end{align*}
	with $F_{-j} = (-1)^{j-1}F_j$ and $L_{-j} = (-1)^j L_j$. The Binet formulas for these sequences are
	\begin{equation}\label{Binet}
		F_j = \frac{\alpha^j - \beta^j}{\alpha - \beta}, \qquad L_j = \alpha^j + \beta^j, \quad j\in\mathbb Z,
	\end{equation}
	with $\alpha=\frac{1+\sqrt 5}2$ being the golden ratio and $\beta=-\frac{1}{\alpha}$.
	They are indexed as sequences {A000045} and {A000032} in the On-Line Encyclopedia of Integer Sequences \cite{OEIS}.
	Koshy \cite{Koshy} and Vajda \cite{Vajda} have written excellent books on these sequences. The gibonacci sequence
	$(G_j)_{j\in\mathbb Z}$ is a slight generalization of $F_j$ and $L_j$. It has the same recurrence relation as the Fibonacci sequence
	but starts with arbitrary initial values, i.e.,
	\begin{equation*}
		G_j  = G_{j - 1} + G_{j - 2},\quad (j \ge 2),
	\end{equation*}
	with $G_0$ and $G_1$ arbitrary numbers (usually integers) not both zero.
	
	Harmonic numbers $H_z$ and odd harmonic numbers $O_z$ are defined
	for $0\ne z\in\mathbb C\setminus\mathbb Z^{-}$ by the recurrence relations
	\begin{equation*}
		H_z = H_{z - 1} + \frac{1}{z} \qquad \text{and} \qquad O_z = O_{z - 1} + \frac{1}{2z - 1},
	\end{equation*}
	with $H_0=0$ and $O_0=0$. Harmonic numbers are connected to the digamma function $\psi(z)=\Gamma'(z)/\Gamma(z)$ through the fundamental relation
	\begin{equation}\label{Harm_psi}
		H_z = \psi(z + 1) + \gamma.
	\end{equation}
	
	Generalized harmonic numbers $H_z^{(m)}$ and odd harmonic numbers $O_z^{(m)}$ of order $m\in\mathbb C$ are defined by
	\begin{equation*}
		H_z^{(m)} = H_{z - 1}^{(m)} + \frac{1}{z^m} \qquad \text{and} \qquad
		O_z^{(m)} = O_{z - 1}^{(m)} + \frac{1}{(2z - 1)^m},
	\end{equation*}
	with $H_0^{(m)}=0$ and $O_0^{(m)}=0$ so that $H_z=H_z^{(1)}$ and $O_z=O_z^{(1)}$. The recurrence relations imply
	that if $z=n$ is a non-negative integer, then
	\begin{equation*}
		H_n^{(m)} = \sum_{j = 1}^n \frac{1}{j^m} \qquad \text{and} \qquad O_n^{(m)} = \sum_{j = 1}^n \frac{1}{(2j - 1)^m}.
	\end{equation*}
	
	\par Exploring identities that combine Fibonacci numbers and harmonic numbers offers a unique intersection between combinatorics and analysis. While each sequence has been extensively studied on its own, their combination reveals surprising and elegant mathematical structures. Investigating such identities deepens our understanding of number theory and enriches the study.
	\par There is a growing body of work dedicated to identities involving Fibonacci and harmonic numbers. 
	In \cite{Batir21}, Bat\i r and Sofo presented many Fibonacci-harmonic sums and we state two beautiful sums from their paper.
	\begin{align*}
		&\sum_{k=1}^n \binom{n}{k} H_k F_{3k} = 2^n \left(H_n F_{2n} - \sum_{k=1}^n \frac{F_{2n-2k}}{k2^k}\right), \\
		&\sum_{k=1}^n (-1)^k \binom{n}{k} H_k F_{3k} = (-1)^n 2^n \left(H_n F_n - \sum_{k=1}^n \frac{(-1)^k F_{n-k}}{k^2 2^k}\right).
	\end{align*}
	Also, in \cite{Frontczak20} Frontczak presented the following;
	\begin{align*}
		&\sum_{k=0}^n \binom{n}{k} k F_k H_k = n F_{2n-1} H_n - \sum_{k=0}^{n-1} \left(k F_{2k}+F_{2k+1}\right) H_{n-1-k}, \\
		&\sum_{k=0}^n \binom{n}{k} k L_k H_k = n L_{2n-1} H_n - \sum_{k=0}^{n-1} \left(k L_{2k}+L_{2k+1}\right) H_{n-1-k}.
	\end{align*}
	\par Although significant progress has been made in evaluating series that involve Fibonacci and harmonic numbers, many avenues remain largely unexplored, especially those concerning generalized harmonic numbers, alternating series, and variations of the Fibonacci sequence. This paper contributes to the ongoing research by establishing new identities and summation formulas that interweave these classical sequences in fresh and insightful ways. By employing three different techniques, we reveal deeper connections that underscore the elegance and depth inherent in Fibonacci-harmonic sums. A sample of our new results without binomial coefficients is given below:
	\begin{align*}
	&\sum_{k=1}^n \frac{(-1)^{k}}{F_{k} F_{k+1}} H_{k} = \sum_{k=1}^n \frac{F_{k}}{F_{k+1}}\frac{1}{k+1} - \frac{F_{n}}{F_{n+1}} H_{n+1}, \\
	&\sum_{k=1}^n \frac{F_{4k+2}}{F_{2k}^2 F_{2k+2}^2} H_{k} = \sum_{k=1}^n \frac{1}{F_{2k+2}^2}\frac{1}{k+1} - \frac{H_{n+1}}{F_{2n+2}^2} + 1,\\
	&\sum_{k = 0}^n H_k F_{n - k}^2 = \sum_{k = 0}^n \frac{{F_{n - k} F_{n - k - 1} }}{{k + 1}}, \\
	&\sum_{k = 0}^n H_k G_{k + 1}^2 = H_{n + 1} G_{n + 1} G_{n + 2} - \sum_{k = 0}^n \frac{G_{k + 1} G_{k + 2}}{k + 1}.
	\end{align*}

	\section{Required identities}

	\begin{lemma}\label{lem.integ}
	For $u,v\in\mathbb C\setminus\mathbb Z^{-}$, we have
	\begin{align}
	\int_0^1 x^u \left( {1 - x} \right)^v dx &= \frac{1}{{\binom{{u + v + 1}}{{u + 1}}\left( {u + 1} \right)}},\label{beta}\\
	\int_0^1 x^u \left( {1 - x} \right)^v \ln x\,dx &= -\frac{{H_{u + v + 1} - H_u}}{{\binom{{u + v + 1}}{{u + 1}} (u + 1)}}\label{beta_der},\\
	\int_0^1 x^u \left( {1 - x} \right)^v \ln ^2 x\,dx &= \frac{{\left( {H_{u + v - 1} - H_u } \right)^2 + H_{u + v + 1}^{(2)} - H_u^{(2)} }}{{		\binom{{u + v + 1}}{{u + 1}}(u + 1)}}\label{beta_der2},
	\end{align}
	and, in particular,
	\begin{equation}\label{eq.fcdh8uy}
	\int_0^1 x^{u - 1} \ln (1 - x) dx = - \frac{{H_u}}{u}, \qquad 0\ne u\in\mathbb C\setminus\mathbb Z^{-}.
	\end{equation}
	\end{lemma}
	\begin{proof}
	The identity~\eqref{beta} is the well-known Beta function integral while~\eqref{beta_der} is obtained by differentiating with respect 
	to $u$. Identity~\eqref{beta_der2} follows from differentiating~\eqref{beta_der} with respect to $u$. These identities are also entries of 	 Gradshteyn and Ryzhik's book~\cite{GrRy07} and other tables.
	\end{proof}
	
	\begin{lemma}\label{lem.ho}
	If $n$ is an integer, then
	\begin{align}
			H_{n - 1/2} &= 2O_n - 2\ln 2 \label{eq.zts4fm1} \\
			H_{n - 1/2} - H_{ - 1/2} &= 2O_n, \label{eq.plh634k} \\
			H_{n - 1/2} - H_{1/2} &= 2\left( {O_n  - 1} \right), \label{eq.hgplrbd}\\
			H_{n + 1/2} - H_{ - 1/2} &= 2O_{n + 1}, \label{eq.ivi1ex5} \\
			H_{n + 1/2} - H_{1/2} &= 2\left( {O_{n + 1}  - 1} \right), \label{eq.u6ng5d6}\\
			H_{n + 1/2} - H_{n - 1/2} &= \frac{2}{{2n + 1}},\\
			H_{n - 1/2} - H_{-3/2} &= 2\left( {O_n  - 1} \right) \label{eq.pobmr6h},\\
			H_{n + 1/2} - H_{-3/2} &= 2\left( {O_{n + 1} - 1} \right).
	\end{align}
	\end{lemma}
	\begin{proof}
	Use \eqref{Harm_psi} as the definition of the harmonic numbers for all complex $n$ (excluding zero and the negative integers) 
	and use the known result for the digamma function at half-integer arguments, namely,
	$$\psi(n+1/2) = -\gamma - 2\ln 2 + \sum_{k = 1}^n \frac{1}{2k-1}.$$
	\end{proof}
	
	\begin{lemma}\label{lem.binomial}
		We have
		\begin{align}
			\binom{{r + 1/2}}{s} &= \frac{{\frac{{2r + 1}}{{2s}}\binom{{2s}}{s}}}{{\binom{{r}}{s}2^{2s} }},\\
			\binom{{1/2}}{r} &= \left( { - 1} \right)^{r + 1} \frac{{\binom{{2r}}{r}}}{{2^{2r} \left( {2r - 1} \right)}},\\
			\binom{{r - 1/2}}{s} &= \binom{{2r}}{r}\binom{{r}}{s}\frac{1}{{\binom{{2\left( {r - s} \right)}}{{r - s}}2^{2s} }},\label{eq.qdgxupm}\\
			\binom{{ - 1/2}}{r} &= ( - 1)^r \binom{{2r}}{r}2^{ - 2r} .
		\end{align}
	\end{lemma}
	\begin{proof}
		These consequences of the generalized binomial coefficients, are easy to derive using the Gamma function. They can also be found in
		Gould's book \cite{Gould}. We will impose the various restrictions later.
	\end{proof}
	
	\section{Method I: Identities based on Abel's partial summation formula}
	
	The approach here is, in each case, to use either a basic Fibonacci identity or a combinatorial identity in conjunction with Abel's partial summation formula.
	\begin{lemma}
		Let ${(a_j)}_{j\in\mathbb Z}$ and ${(b_j)}_{j\in\mathbb Z}$ be sequences. Then
		\begin{gather}
			\sum_{k = 1}^n b_k (a_{k + 1} - a_k) = \sum_{k = 1}^n a_{k + 1} (b_k - b_{k + 1}) + a_{n + 1} b_{n + 1} - a_1 b_1, \label{abel1} \\
			\sum_{k = 1}^n b_k (a_{k + 1} + a_k) = \sum_{k = 1}^n a_{k + 1} (b_k + b_{k + 1}) - a_{n + 1} b_{n + 1} + a_1 b_1. \label{abel2}
		\end{gather}
	\end{lemma}
	\begin{proof}
		The proof of these very useful identities is surprisingly simple. One merely needs to observe that
		\begin{equation*}
			\sum_{k = 1}^n a_{k + 1} b_{k + 1} = \sum_{k = 1}^n a_k b_k + a_{n + 1} b_{n + 1} - a_1 b_1.
		\end{equation*}
	\end{proof}
	
	\subsection{Identities derived from basic Fibonacci identities}
	
	Most of the identities used in this section come from Vajda's classic \cite{Vajda}.
	We begin with two theorems that follow essentially from the definition of Fibonacci numbers.
	
	\begin{theorem}
		For $r\in\mathbb{Z}$ and $s,t\in \mathbb{C}\setminus\mathbb{Z}^{-}$ such that $s+t\notin\mathbb{Z}^{-}$, we have
		\begin{equation}\label{thm1_eq1}
			\sum_{k = 1}^n \frac{F_{k+r+2}}{\binom{t+k+s}{t+1}} = \sum_{k = 1}^n \frac{F_{k+r+1}}{\binom{t+k+s}{t+1}}
			- \frac{F_{n+r+1}}{\binom{t+n+s+1}{t+1}} + \sum_{k = 1}^n \frac{F_{k+r+1}}{\binom{t+k+s+1}{t+1}} + \frac{F_{r+1}}{\binom{s+t+1}{t+1}}
		\end{equation}
		Thus, for $n$ a non-negative integer, we have
		\begin{equation}\label{thm1_eq2}
			\sum_{k=1}^n\frac{F_{k+r+2}}{k+s} = \sum_{k=1}^n \frac{(2k+3)F_{k+r+1}}{(k+s)(k+s+1)} + \frac{F_{r+1}}{s+1} - \frac{F_{n+r+1}}{n+s+1} .
		\end{equation}
	\end{theorem}
	
	\begin{proof}
		Since we know that $F_{k+1}-F_k=F_{k-1}$, we set $a_k=F_{k+r}$ and $b_k=(1-x)^k$ in \eqref{abel2}, to obtain
		\begin{equation*}
			\sum_{k=1}^n (1-x)^k F_{k+r+2} = \sum_{k=1}^n (1-x)^k F_{k+r+1} + \sum_{k=1}^n (1-x)^{k+1} F_{k+r+1} - (1-x)^{n+1} F_{n+r+1} + (1-x)F_{r+1}.
		\end{equation*}
		Multiply through by $x^t(1-x)^{s-1}$ and then by term-wise integration, we obtain \eqref{thm1_eq1} directly on account
		of Lemma \ref{lem.integ}. Setting $(s,t)=(0,0)$ in \eqref{thm1_eq1} gives rise to \eqref{thm1_eq2}.
	\end{proof}
	
	\begin{theorem}
		For $r\in \mathbb{Z}$ and $s,t\in \mathbb{C}\setminus \mathbb{Z}^{-}$ such that $s+t\notin \mathbb{Z}^{-}$,
		\begin{equation}\label{thm2_eq1}
			\begin{split}
				\sum_{k = 1}^n \frac{H_{t}-H_{k+t+s-1}}{\binom{t+k+s-1}{t+1}}F_{k+r-1} &=
				\sum_{k = 1}^n \frac{H_{t}-H_{k+t+s-1}}{\binom{t+k+s-1}{t+1}}F_{k+r+1} - \sum_{k = 1}^n \frac{H_{t}-H_{k+t+s}}{\binom{t+k+s}{t+1}}F_{k+r+1} \\ &\qquad + \frac{H_{t}-H_{n+t+s}}{\binom{n+t+s}{t+1}}F_{n+r+1} - \frac{H_{t}-H_{s+t}}{\binom{t+s}{t+1}}F_{r+1}.
			\end{split}
		\end{equation}
	\end{theorem}
	\begin{proof}
		In~\eqref{abel1}, set $a_k=F_{k+r}$ and $b_k=(1-x)^k$, multiply through by $x^t(1-x)^{s-2}\ln x$ and integrate from 0 to 1. The result is
		\eqref{thm2_eq1} using \eqref{beta_der}. 
	\end{proof}

	\begin{theorem}
		For $n$ a positive integer, we have
		\begin{equation}\label{eq.xslb526}
			\sum_{k = 1}^n \frac{{H_k}}{k} F_{k - 2} F_{k + 1} = \sum_{k = 1}^{n - 1} \frac{{H_{k + 1} - 1}}{k(k + 1)} F_k^2 + \frac{{H_n}}{n} F_n^2,
		\end{equation}
		\begin{equation}\label{eq.qsl14h7}
			\sum_{k = 1}^n \frac{2^{2k}}{k}\frac{O_k}{\binom{{2k}}{k}} F_{k - 2} F_{k + 1} = 2\sum_{k = 1}^{n - 1} \frac{2^{2k}}{k(k + 1)}\frac{{\left( {O_{k + 1}  - 1} \right)}}{{\binom{{2\left( {k + 1} \right)}}{{k + 1}}}}F_k^2 + \frac{{2^{2n} }}{n}\frac{{O_n }}{{\binom{{2n}}{n}}}F_n^2 ,
		\end{equation}
		and more generally, for $r\in\mathbb Z$ and $s,t\in\mathbb C\setminus\mathbb Z^{-}$ such that $s+t\not\in\mathbb Z^{-}$,
		\begin{equation}\label{eq.c6bw3gq}
			\begin{split}
				\sum_{k = 1}^n {\frac{{H_{k + s + t}  - H_t }}{{\binom{{k + s + t}}{{t + 1}}}}F_{k + r - 1} F_{k + r + 2} }  &= \frac{{t + 1}}{{t + 2}}\sum_{k = 1}^n {\frac{{H_{k + s + t}  - H_{t + 1} }}{{\binom{{k + s + t}}{{t + 2}}}}F_{k + r}^2 }  - \frac{{H_{s + t}  - H_t }}{{\binom{{s + t}}{{t + 1}}}}F_{r + 1}^2\\
				&\qquad\qquad + \frac{{H_{n + s + t}  - H_t }}{{\binom{{n + s + t}}{{t + 1}}}}F_{n + r + 1}^2 .
			\end{split}
		\end{equation}
	\end{theorem}
	\begin{proof}
		Here, the basic Fibonacci identity is
		\begin{equation*}
			F_{k - 1} F_{k + 2} = F_{k + 1}^2 - F_k^2.
		\end{equation*}
		In~\eqref{abel1}, set $a_k=F_{k+r}^2$ and $b_k=(1-x)^k$ to obtain
		\begin{equation*}
			\sum_{k = 1}^n (1 - x)^{k + 1} F_{k + r - 1} F_{k + r + 2} = \sum_{k = 1}^n (1 - x)^k x F_{k + r}^2 - (1 - x)F_{r + 1}^2
			+ (1 - x)^{n + 1} F_{n + r + 1}^2,
		\end{equation*}
		which upon multiplying through by $x^t(1-x)^{s-2}$ gives
		\begin{equation*}
			\begin{split}
				\sum_{k = 1}^n {x^t \left( {1 - x} \right)^{k + s - 1} F_{k + r - 1} F_{k + r + 2} }  &= \sum_{k = 1}^n {x^{t + 1} \left( {1 - x} \right)^{k + s - 2} F_{k + r}^2 }  - x^t \left( {1 - x} \right)^{s - 1} F_{r + 1}^2\\
				&\qquad + x^t \left( {1 - x} \right)^{s + n - 1} F_{n + r + 1}^2 ;
			\end{split}
		\end{equation*}
		from which term-wise integration produces~\eqref{eq.c6bw3gq} on account of Lemma~\ref{lem.integ}. Identities~\eqref{eq.xslb526} and~\eqref{eq.qsl14h7} are special cases of~\eqref{eq.c6bw3gq}, being evaluations respectively at $(s=0=t)$ and \mbox{$(s=0,t=-1/2)$}. Note the use of Lemmata~\ref{lem.ho} and~\ref{lem.binomial}.
		
		In deriving~\eqref{eq.xslb526} and~\eqref{eq.qsl14h7} we used the following limits:
		\begin{gather*}
			\lim_{s\to 0}\frac{H_{s + 1} - 1}{s} = \frac{\pi^2}{6} - 1,\\
			\lim_{r\to 0}\frac{O_{r + 1} - 1}{r} = \frac{\pi^2}{4} - 2.
		\end{gather*}
		which are easily derived by converting to the digamma function and using L'Hospital's rule.
	\end{proof}
	
	\begin{theorem}\label{thm1_Fib_squared}
		For all $r\in\mathbb Z$ and $s,t\in\mathbb C\setminus\mathbb Z^{-}$ such that $s+t\not\in\mathbb Z^{-}$, we have
		\begin{equation}
			\sum_{k = 1}^n \frac{F_{4(k+r)+2}}{\binom{k + t + s - 1}{t + 1}} =
			\frac{t + 1}{t + 2} \sum_{k = 1}^n \frac{F_{2(k+r)+2}^2}{\binom{k + t + s}{t + 2}} + \frac{F_{2(n+r)+2}^2}{\binom{n + t + s}{t + 1}}
			- \frac{F_{2(r+1)}^2}{\binom{t + s}{t + 1}}.
		\end{equation}
		In particular, we have for all $r\in\mathbb Z$ and $s\geq 1$
		\begin{equation}
			\sum_{k = 1}^n \frac{F_{4(k+r)+2}}{k + s - 1} = \sum_{k = 1}^n \frac{F_{2(k+r)+2}^2}{(k+s)(k+s-1)} + \frac{F_{2(n+r)+2}^2}{n+s}
			- \frac{F_{2(r+1)}^2}{s}.
		\end{equation}
	\end{theorem}
	\begin{proof}
		The basic Fibonacci identity is
		\begin{equation*}
			F_{2k+2}^2 - F_{2k}^2 = F_{4k+2},
		\end{equation*}
		which is a special case of an identity of Howard \cite{Howard03}
		\begin{equation}\label{Howard_id}
			F_{a+b}^2 - F_{a-b}^2 = F_{2a} F_{2b}.
		\end{equation}
		Hence, we choose in \eqref{abel1} $a_k=F_{2(k+r)}^2$ and $b_k=(1-x)^k$. Then, after multiplying through by $x^t (1-x)^{s-2}$ we get
		\begin{equation*}
			\begin{split}
				\sum_{k = 1}^n x^t (1-x)^{k + s - 2} F_{4(k + r) + 2} &= \sum_{k = 1}^n F_{2(k + r) + 2}^2 x^{t+1} (1-x)^{k+s-2}
				+ F_{2(n + r) + 2}^2 x^{t} (1-x)^{k+s-1} \\
				&\qquad - F_{2(r + 1)}^2 x^t (1-x)^{s-1}.
			\end{split}
		\end{equation*}
		The result is now a consequence of \eqref{beta}.
	\end{proof}
	A Fibonacci-harmonic identity is obtained by exchanging \eqref{beta} by \eqref{beta_der}.
	
	\begin{theorem}\label{thm2_Fib_squared}
		For all $r\in\mathbb Z$ and $s,t\in\mathbb C\setminus\mathbb Z^{-}$ such that $s+t\not\in\mathbb Z^{-}$, we have
		\begin{equation}
			\begin{split}
				\sum_{k = 1}^n \frac{H_t - H _{k + t + s - 1}}{\binom{k + t + s - 1}{t + 1}} F_{4(k+r)+2} &=
				\frac{t + 1}{t + 2} \sum_{k = 1}^n \frac{H_{t+1} - H_{k + t + s}}{\binom{k + t + s}{t + 2}} F_{2(k+r)+2}^2 \\
				&\qquad + \frac{H_t - H_{n + t + s}}{\binom{n + t + s}{t + 1}} F_{2(n+r)+2}^2 - \frac{H_t - H_{t + s}}{\binom{t + s}{t + 1}} F_{2(r+1)}^2.
			\end{split}
		\end{equation}
		In particular, we have for all $r\in\mathbb Z$ and $s\geq 1$
		\begin{equation}
			\sum_{k = 1}^n \frac{H_{k+s-1}}{k + s - 1} F_{4(k+r)+2} = \sum_{k = 1}^n \frac{H_{k+s}-1}{(k+s)(k+s-1)} F_{2(k+r)+2}^2
			+ \frac{H_{n+s}}{n+s} F_{2(n+r)+2}^2 - \frac{H_s}{s} F_{2(r+1)}^2.
		\end{equation}
	\end{theorem}
	
	\begin{remark}
		The Lucas counterparts of Theorems \ref{thm1_Fib_squared} and \ref{thm2_Fib_squared} are obtained straightforwardly using another
		identity of Howard \cite{Howard03}, namely,
		\begin{equation*}
			L_{a+b}^2 - L_{a-b}^2 = 5 F_{2a} F_{2b}.
		\end{equation*}
	\end{remark}
	
	We continue with some sums involving reciprocals of products of Fibonacci numbers.
	
	\begin{theorem}\label{thm1_rec}
		For all $r\geq 1$ and $s\geq 0$ we have
		\begin{equation}\label{eq1_thm1_rec}
			\sum_{k=1}^n \frac{(-1)^{rk}}{F_{rk} F_{r(k+1)}} H_{k+s-1} = \frac{1}{F_r^2} \left ( \sum_{k=1}^n \frac{F_{r(k+2)}}{F_{r(k+1)}}\frac{1}{k+s}
			- \frac{F_{r(n+2)}}{F_{r(n+1)}} H_{n+s} + L_r H_s \right ).
		\end{equation}
		In particular,
		\begin{equation}
			\sum_{k=1}^n \frac{(-1)^{k}}{F_{k} F_{k+1}} H_{k} = \sum_{k=1}^n \frac{F_{k}}{F_{k+1}}\frac{1}{k+1} - \frac{F_{n}}{F_{n+1}} H_{n+1}.
		\end{equation}
	\end{theorem}
	\begin{proof}
		The Fibonacci identity
		\begin{equation*}
			(-1)^t F_a F_b = F_{t+a}F_{t+b} - F_t F_{t+a+b}
		\end{equation*}
		produces after making the replacements $t\mapsto rk$ and $a=b=r$ the identity
		\begin{equation*}
			\frac{F_{r(k+2)}}{F_{r(k+1)}} - \frac{F_{r(k+1)}}{F_{rk}} = (-1)^{rk+1} \frac{F_r^2}{F_{rk} F_{r(k+1)}}.
		\end{equation*}
		This shows that we can choose $a_k=F_{r(k+1)}/F_{rk}$ and $b_k=H_{k+s-1}$, and the result follows from \eqref{abel1}.
	\end{proof}
	
	There also exists an expression for the left hand side of \eqref{eq1_thm1_rec} involving Lucas numbers, which we state in the next theorem.
	
	\begin{theorem}\label{thm2_rec}
		For all $r\geq 1$ and $s\geq 0$
		\begin{equation}\label{eq2_thm1_rec}
			\sum_{k=1}^n \frac{(-1)^{rk}}{F_{rk} F_{r(k+1)}} H_{k+s-1} = \frac{1}{2 F_r} \left ( \sum_{k=1}^n \frac{L_{r(k+1)}}{F_{r(k+1)}}\frac{1}{k+s}
			- \frac{L_{r(n+1)}}{F_{r(n+1)}} H_{n+s} + \frac{L_r}{F_r} H_s \right ).
		\end{equation}
	\end{theorem}
	\begin{proof}
		Here, the relevant Fibonacci identity is
		\begin{equation*}
			F_a L_b - L_{a} F_{b} = 2 (-1)^b F_{a-b}.
		\end{equation*}
		After making the replacements $a\mapsto rk$ and $b\mapsto r(k+1)$, and keeping in mind that $F_{-n}=(-1)^{n+1}F_n$ we get the identity
		\begin{equation*}
			\frac{L_{r(k+1)}}{F_{r(k+1)}} - \frac{L_{rk}}{F_{rk}} = (-1)^{rk+1} \frac{2 F_r}{F_{rk} F_{r(k+1)}}.
		\end{equation*}
		This shows that we can choose $a_k=L_{rk}/F_{rk}$ and $b_k=H_{k+s-1}$, and the result follows from \eqref{abel1}.
	\end{proof}
	
	\begin{corollary}
		For all $r\geq 1$ and $s\geq 0$ we have the relation
		\begin{equation}
			\sum_{k=1}^n \frac{L_{r(k+1)}}{F_{r(k+1)}}\frac{1}{k+s} = \frac{2}{F_r} \sum_{k=1}^n \frac{F_{r(k+2)}}{F_{r(k+1)}}\frac{1}{k+s}
			+ \left (L_{r(n+1)} - \frac{2}{F_r} F_{r(n+2)} \right ) \frac{H_{n+s}}{F_{r(n+1)}} + \frac{L_r}{F_r} H_s.
		\end{equation}
	\end{corollary}
	
	\begin{remark}
		The Lucas counterpart of Theorem \ref{thm1_rec} can be derived using the identity
		\begin{equation*}
			5 (-1)^{t+1} F_a F_b = L_{t+a}L_{t+b} - L_t L_{t+a+b}.
		\end{equation*}
	\end{remark}
	
	\begin{theorem}\label{thm3_rec}
		For all odd $r\geq 1$ and all $s\geq 0$ we have
		\begin{equation}\label{eq1_thm3_rec}
			\sum_{k=1}^n \frac{F_{r(2k+1)}}{F_{2rk} F_{2r(k+1)}} H_{k+s-1} = \frac{1}{L_r} \left ( \sum_{k=1}^n \frac{1}{F_{2r(k+1)}}\frac{1}{k+s}
			- \frac{H_{n+s}}{F_{2r(n+1)}}  + \frac{H_s}{F_{2r}} \right ).
		\end{equation}
		In particular,
		\begin{equation}
			\sum_{k=1}^n \frac{F_{2k+1}}{F_{2k} F_{2k+2}} H_{k} = \sum_{k=1}^n \frac{1}{F_{2k+2}}\frac{1}{k+1} - \frac{H_{n+1}}{F_{2n+2}} +1.
		\end{equation}
	\end{theorem}
	\begin{proof}
		The basic Fibonacci identity that we need is
		\begin{equation*}
			L_a F_b = F_{a+b} + (-1)^a F_{b-a}.
		\end{equation*}
		Now, with $a=r$ odd and $b=r(2k+1)$ gives
		\begin{equation*}
			\frac{1}{F_{2r(k+1)}} - \frac{1}{F_{2rk}} = - \frac{L_r F_{r(2k+1)}}{F_{2rk} F_{2r(k+1)}}.
		\end{equation*}
		This shows that we can choose $a_k=1/F_{2rk}$, $r$ odd, and $b_k=H_{k+s-1}$, and the result follows from \eqref{abel1}.
	\end{proof}
	
	We note that the identities of Howard can also be applied to produce reciprocal sums involving squares of Fibonacci numbers.
	
	\begin{theorem}\label{thm4_rec}
		For all even $r\geq 2$ and all $s\geq 0$ we have
		\begin{equation}\label{eq1_thm4_rec}
			\sum_{k=1}^n \frac{F_{r(2k+1)}}{F_{rk}^2 F_{r(k+1)}^2} H_{k+s-1} = \frac{1}{F_r} \left ( \sum_{k=1}^n \frac{1}{F_{r(k+1)}^2}\frac{1}{k+s}
			- \frac{H_{n+s}}{F_{r(n+1)}^2}  + \frac{H_s}{F_{r}^2} \right ).
		\end{equation}
		In particular,
		\begin{equation}
			\sum_{k=1}^n \frac{F_{4k+2}}{F_{2k}^2 F_{2k+2}^2} H_{k} = \sum_{k=1}^n \frac{1}{F_{2k+2}^2}\frac{1}{k+1} - \frac{H_{n+1}}{F_{2n+2}^2} +1.
		\end{equation}
	\end{theorem}
	\begin{proof}
		Apply the first of Howard's identities \eqref{Howard_id} with $a=rk+r/2$, $r$ even, and $b=r/2$ to get
		\begin{equation*}
			\frac{1}{F_{r(k+1)}^2} - \frac{1}{F_{rk}^2} = - \frac{F_r F_{r(2k+1)}}{F_{rk}^2 F_{r(k+1)}^2}.
		\end{equation*}
		Choose $a_k=1/F_{rk}^2$, $r$ even, and $b_k=H_{k+s-1}$, and the result follows from \eqref{abel1}.
	\end{proof}
	
	\begin{theorem}\label{thm5_rec}
		For all even $r\geq 2$ and all $s\geq 0$ we have
		\begin{equation}\label{eq1_thm5_rec}
			\sum_{k=1}^n \frac{F_{2r(k+1)}}{F_{rk} F_{r(k+1)}^2 F_{r(k+2)}} H_{k+s-1} = \frac{1}{F_r} \left ( \sum_{k=1}^n \frac{1}{F_{r(k+1)} F_{r(k+2)}}\frac{1}{k+s} - \frac{H_{n+s}}{F_{r(n+1)} F_{r(n+2)}} + \frac{H_s}{F_{r} F_{2r}} \right ).
		\end{equation}
		In particular,
		\begin{equation}
			\sum_{k=1}^n \frac{F_{4k+4}}{F_{2k} F_{2k+2}^2 F_{2k+4}} H_{k} = \sum_{k=1}^n \frac{1}{F_{2k+2} F_{2k+4}}\frac{1}{k+1}
			- \frac{H_{n+1}}{F_{2n+2} F_{2n+4}} + \frac{1}{3}.
		\end{equation}
	\end{theorem}
	\begin{proof}
		Howard \cite{Howard03} derived the following identity valid for indices $a,b$ and $c$:
		\begin{equation}\label{Howard_id2}
			F_a F_{a+2b+c} = F_{a+b+c} F_{a+b} + (-1)^{a+1} F_b F_{b+c}.
		\end{equation}
		This is equivalent to
		\begin{equation*}
			\frac{F_{a} F_{a+2b+c}}{F_{a+b} F_b F_{b+c} F_{a+b+c}} = \frac{1}{F_b F_{b+c}} + (-1)^{a+1} \frac{1}{F_{a+b} F_{a+b+c}},
		\end{equation*}
		or, with $a$ being even,
		\begin{equation*}
			\frac{F_{a} F_{a+2b+c}}{F_{a+b} F_b F_{b+c} F_{a+b+c}} = \frac{1}{F_b F_{b+c}} - \frac{1}{F_{a+b} F_{a+b+c}}.
		\end{equation*}
		Making the replacements $a=c=r$ and $b=rk$ yields
		\begin{equation*}
			\frac{1}{F_{r(k+1)} F_{r(k+2)}} - \frac{1}{F_{rk} F_{r(k+1)}} = - \frac{F_r F_{2r(k+1)}}{F_{rk} F_{r(k+1)}^2 F_{r(k+2)}}.
		\end{equation*}
		Choose $a_k=1/(F_{rk} F_{r(k+1)})$, $r$ even, and $b_k=H_{k+s-1}$, and the result follows from \eqref{abel1}.
	\end{proof}
	
	\begin{remark}
		A similar identity involving products of Fibonacci and Lucas numbers can be derived straightforwardly using the complement
		of \eqref{Howard_id2} from \cite{Howard03}
		\begin{equation}
			F_a L_{a+2b+c} = L_{a+b+c} F_{a+b} + (-1)^{a+1} F_b L_{b+c}.
		\end{equation}
	\end{remark}
	
	\begin{theorem}
		We have
		\begin{gather}
		\sum_{k = 0}^n {H_k F_{n - k}^2 } = \sum_{k = 0}^n {\frac{{F_{n - k} F_{n - k - 1} }}{{k + 1}}},\\
		\sum_{k = 0}^n {H_k G_{k + 1}^2 } = H_{n + 1} G_{n + 1} G_{n + 2} - \sum_{k = 0}^n {\frac{{G_{k + 1} G_{k + 2} }}{{k + 1}}},
		\end{gather}
		and more generally, if $0\ne s\in\mathbb C\setminus\mathbb Z^{-}$, $m$ and $n$ are non-negative integers and $r$ is any integer, then
		\begin{equation}\label{eq.g8er2qv}
		\begin{split}
		& - \sum_{k = 0}^n {\frac{1}{{k + s}}\prod_{j = 0}^{2m - 1} {G_{k + r + j + 1} } } + H_{n + s} \prod_{j = 0}^{2m - 1} {G_{n + r + j + 1} }  - H_{s - 1} \prod_{j = 0}^{2m - 1} {G_{j + r} } \\
		&\qquad\qquad =
		\begin{cases}
		F_m \sum_{k = 0}^n {H_{k + s - 1} \left( {G_{k + m + r - 1}  + G_{k + m + r + 1} } \right)\prod_{j = 1}^{2m - 1} {G_{k + r + j} } },&\text{if $m$ is even;}  \\
		L_m \sum_{k = 0}^n {H_{k + s - 1} G_{k + m + r} \prod_{j = 1}^{2m - 1} {G_{k + r + j} } },&\text{if $m$ is odd.}  \\
		\end{cases}
		\end{split}
		\end{equation}
	\end{theorem}
	\begin{proof}
		Let
		\begin{equation}\label{eq.yoh2fwv}
			a_k  = G_{k + r + 2m - 1} G_{k + r + 2m - 2}  \cdots G_{k + r}  = \prod_{j = 0}^{2m - 1} {G_{k + r + j} },\quad b_k=H_{k+s-1}.
		\end{equation}
		Thus
		\begin{equation}\label{eq.if75zm8}
			b_{k + 1}  - b_k  = H_{k + s}  - H_{k + s - 1} = \frac{1}{{k + s}},
		\end{equation}
		and
		\begin{equation}\label{eq.ohpxkfx}
			\begin{split}
				a_{k + 1} - a_k&=\left( {G_{k + r + 2m} - G_{k + r} } \right)\prod_{j = 1}^{2m - 1} {G_{k + r + j} }\\
				&=\begin{cases}
					F_m \left( {G_{k + m + r - 1} + G_{k + m + r + 1} } \right)\prod_{j = 1}^{2m - 1} {G_{k + r + j} } ,&\text{if $m$ is even;}  \\
					L_m G_{k + m + r} \prod_{j = 1}^{2m - 1} {G_{k + r + j} } ,&\text{if $m$ is odd;}  \\
				\end{cases}
			\end{split}
		\end{equation}
		since
		\begin{equation*}
			G_{p + q} + (- 1)^q G_{p - q} = L_q G_p
		\end{equation*}
		and
		\begin{equation*}
			G_{p + q} - (- 1)^q G_{p - q} = F_q \left( G_{p - 1} + G_{p + 1} \right).
		\end{equation*}
		Identity~\eqref{eq.g8er2qv} now follows upon using~\eqref{eq.yoh2fwv},~\eqref{eq.if75zm8} and~\eqref{eq.ohpxkfx} in~\eqref{abel1}.
	\end{proof}

	\subsection{Identities derived from combinatorial identities}

	\begin{theorem}\label{thm.p5cp11r}
		For $n$ a non-negative integer, we have
		\begin{equation}\label{eq.lt63xqk}
			\begin{split}
				\sum_{k = 1}^n {\left( { - 1} \right)^{k - 1} \frac{{\binom{{2k}}{k}}}{{2^{2k} }}O_k F_{k - 1} }  &= \sum_{k = 1}^n {\left( { - 1} \right)^{k - 1} \frac{{\binom{{2k}}{k}}}{{2^{2k} }}\left( {2k + 1} \right)\left( {O_{k + 1}  - 1} \right)F_{k + 1} } \\
				&\qquad + \left( { - 1} \right)^n \frac{{\binom{{2n}}{n}}}{{2^{2n} }}\left( {2n + 1} \right)\left( {O_{n + 1}  - 1} \right)F_n ,
			\end{split}
		\end{equation}
		and, more generally, for $r-n-s-1\not\in\mathbb Z^{-}$,
		\begin{equation}\label{eq.zmuo9et}
			\begin{split}
				&\sum_{k = 1}^n {\binom{{r}}{{k + s}}\left( {H_r  - H_{r - k - s} } \right)G_{k + t} } \\
				&\qquad= \sum_{k = 1}^n {\binom{{r - 1}}{{k + s}}\left( {H_{r - 1}  - H_{r - k - s - 1} } \right)G_{k + t + 2} }\\
				&\qquad + \binom{{r - 1}}{s}\left( {H_{r - 1}  - H_{r - s - 1} } \right)G_{t + 1}  - \binom{{r - 1}}{{n + s}}\left( {H_{r - 1}  - H_{r - s - n - 1} } \right)G_{n + t + 1} .
			\end{split}
		\end{equation}
	\end{theorem}
	\begin{proof}
		In~\eqref{abel2} use
		\begin{equation*}
			a_k  = \binom{{r - 1}}{{k + s - 1}},\quad b_k=x^k,
		\end{equation*}
		to obtain
		\begin{equation}\label{eq.lem837w}
			\sum_{k = 1}^n {\binom{{r}}{{k + s}}x^k}  = \sum_{k = 1}^n {x^k \left( {1 + x} \right)\binom{{r - 1}}{{k + s}}}  + x\binom{{r - 1}}{s} - x^{n + 1} \binom{{r - 1}}{{n + s}},
		\end{equation}
		since
		\begin{equation*}
			\binom{{r - 1}}{{k + s}} + \binom{{r - 1}}{{k + s - 1}} = \binom{{r}}{{k + s}}.
		\end{equation*}
		Now use $x=\alpha=(1+\sqrt 5)/2$ and $x=\beta=(1-\sqrt 5)/2$, in turn, in~\eqref{eq.lem837w} and add, using the Binet formula to obtain
		\begin{equation}\label{eq.p517m2v}
			\sum_{k = 1}^n {\binom{{r}}{{k + s}}G_{k + t} }  = \sum_{k = 1}^n {\binom{{r - 1}}{{k + s}}G_{k + t + 2} }  + \binom{{r - 1}}{s}G_{t + 1}  - \binom{{r - 1}}{{n + s}}G_{n + t + 1} .
		\end{equation}
		Differentiation of~\eqref{eq.p517m2v} with respect to $r$, using
		\begin{equation*}
			\begin{split}
				\frac{d}{{da}}\binom{{a + b}}{{c + d}} &= \binom{{a + b}}{{c + d}}\left( {\psi \left( {a + b + 1} \right) - \psi \left( {a + b - c - d + 1} \right)} \right)\\
				&\qquad = \binom{{a + b}}{{c + d}}\left( {H_{a + b}  - H _{a + b - c - d} } \right),
			\end{split}
		\end{equation*}
		gives~\eqref{eq.zmuo9et}, of which evaluation at $r=-1/2$ gives~\eqref{eq.lt63xqk} on account of Lemma~\ref{lem.binomial}.
	\end{proof}
	
	\begin{theorem}
	For $r-n-s-1\notin \mathbb{Z}^{-}$
	\begin{equation}
	\begin{split}
	&\sum_{k = 1}^n\binom{r}{k+s}\left(H_{r-k-s}-H_{k+s}\right)G_{k+t}=\sum_{k=1}^n\binom{r-1}{k+s}\left(H_{r-1-k-s}-H_{k+s}\right)G_{k+t+2}\\
	&                 +\binom{r-1}{s}\left(H_{r-1-s}-H_s\right)G_{t+1}-\binom{r-1}{n+s}\left(H_{r-1-n-s}-H_{n+s}\right)G_{n+t+1}.
	\end{split}
	\end{equation}
	\end{theorem}
	\begin{proof}
		This theorem follows directly from differentiating \eqref{eq.p517m2v} with respect to $s$.
	\end{proof}
	\begin{theorem}
		For  $r-n-s-1\notin \mathbb{Z}^{-}$
		\begin{equation}
			\label{49}
			\begin{split}
				&\sum_{k = 1}^n\binom{r}{k+s}\left(H_{r}-H_{r-k-s}\right)G_{3k+t}
				=2\sum_{k=1}^n\binom{r-1}{k+s}\left(H_{r-1}-H_{r-k-s-1}\right)G_{3k+t+2}\\
				&+\binom{r-1}{s}\left(H_{r-1}-H_{r-s-1}\right)G_{t+3}-\binom{r-1}{n+s}\left(H_{r-1}-H_{r-s-n-1}\right)G_{3n+t+3}.
			\end{split}
		\end{equation}
		
	\end{theorem}
	\begin{theorem}
		For  $r-n-s-1\notin \mathbb{Z}^{-}$
		\begin{equation}
			\label{121}
			\begin{split}
				&\sum_{k = 1}^n(-1)^k\binom{r}{k+s}\left(H_{r}-H_{r-k-s}\right)G_{k+t}
				=\sum_{k=1}^n(-1)^{k+1}\binom{r-1}{k+s}\left(H_{r-1}-H_{r-k-s-1}\right)G_{k+t-1}\\
				&-\binom{r-1}{s}\left(H_{r-1}-H_{r-s-1}\right)G_{t+1}+(-1)^n\binom{r-1}{n+s}\left(H_{r-1}-H_{r-s-n-1}\right)G_{n+t+1}.
			\end{split}
		\end{equation}
	\end{theorem}
	\begin{proof}
		We set  $x=\alpha^3$, $x=\beta^3$, in turn, in~\eqref{eq.lem837w} to obtain the result in \eqref{49} and similarly set $x=-\alpha$, 
		$x=-\beta$ to get the result in \eqref{121}.
	\end{proof}
	
	\section{Method II: Identities based on polynomial combinatorial identities}
	
	The Fibonacci-harmonic identities derived in this section are based on making appropriate substitutions in known polynomial identities.
	
	\subsection{Fibonacci-harmonic sums derived from an identity from Gould's book}
	
	\begin{lemma}[Gould~{\cite[Entry (1.9), page 2]{Gould}}]
		If $x$ and $y$ are complex variables, then
		\begin{equation}\label{46}
		\sum_{k = 0}^n \binom{{x}}{k} y^k = \sum_{k = 0}^n \left( {- 1} \right)^k \binom{{n - x}}{k}\left( {1 + y} \right)^{n - k} y^k.
		\end{equation}
	\end{lemma}
	
	\begin{theorem}
		If $n-x-k\notin \mathbb{Z}^{-}$,
		\begin{equation}\label{47}
		\sum_{k=0}^n\binom{x}{k}(H_x-H_{x-k})G_{k+t} = \sum_{k=0}^n (-1)^k \binom{n-x}{k} (H_{n-x}-H_{n-x-k}) G_{2n-k+t}.
		\end{equation}
	\end{theorem}
	\begin{proof}
		Set $y=\alpha$ and $y=\beta$, in turns, in \eqref{46} and use the Binet formula to get
		\begin{equation}\label{48}
			\sum_{k=0}^n\binom{x}{k}G_{k+t} = \sum_{k=0}^n (-1)^k \binom{n-x}{k} G_{2n-k+t}.
		\end{equation}
		Differentiating \eqref{48} with respect to $x$, using
		$$\frac{d}{dx}\binom{x}{k}=\binom{x}{k}(H_x-H_{x-k})$$
		Thus, \eqref{47} follows directly.
	\end{proof}
	
	\begin{corollary}
	We have,
	\begin{equation}
	\sum_{k=0}^n \frac{(-1)^k}{2^{2k}} O_{k+1}G_{k+t}
	= (2n+1) \sum_{k=0}^n \frac{(-1)^{k+1} \binom{2k}{k}} {2^{2k+1}\binom{n}{k}}(O_{n+1}-O_{n-k})G_{2n-k+t}.
		\end{equation}
	\end{corollary}
	\begin{proof}
		Set $x=-1/2$ in \eqref{47}, using (14) and (17). Thus, with the fact that 
		$$H_{n+1/2}-H_{n-k+1/2}=2(O_{n+1}-O_{n-k}),$$
		the result follows immediately.
	\end{proof}
	
	\begin{theorem}
	If $n-x-k\notin\mathbb{Z}^{-}$,
	\begin{equation}
	\sum_{k = 0}^n \binom{x}{k} (H_x-H_{x-k}) G_{k+t}(G_{2k+t}-(-1)^k) = \sum_{k = 0}^n (-1)^k 2^{n-k}\binom{n-x}{k}(H_{n-x}-H_{n-x-k})G_{2n+k+t}.
	\end{equation}
	\end{theorem}
	\begin{proof}
	Set $y=\alpha^3$ and $y=\beta^3$, in turn, in \eqref{46} and differentiating the resulting identity with respect to $x$. Hence, the result.
	\end{proof}
	
	\subsection{Fibonacci-harmonic sums derived from another identity from Gould's book}
	
	\begin{lemma}
	If $b\in\mathbb C\setminus \mathbb{Z}^{-}$ and $x$ is a complex variable, then
	\begin{equation}\label{eq.gldge7l}
	\sum_{k = 0}^n ( - 1)^{n - k} \binom{{n - b}}{{k - b}} H_{k - b} \left( {1 + x} \right)^k 
	= \sum_{k = 0}^n \binom{{b}}{{n - k}}\left( H_b + H_{n - b} - H_{b - n + k} \right)x^k.
	\end{equation}
	\end{lemma}
	\begin{proof}
		Differentiate the following identity \cite[Equation (1.10)]{Gould}
		\begin{equation}
		\sum_{k = 0}^n ( - 1)^{n - k} \binom{{n - b}}{{k - b}}\left( {1 + x} \right)^k = \sum_{k = 0}^n \binom{{b}}{{n - k}}x^k
		\end{equation}
		with respect to $b$.
	\end{proof}
	
	\begin{theorem}
	For $n\in \mathbb{N}$ and $b\in \mathbb{C}\setminus\mathbb{Z}^{-}$ and $x$ is a complex variable, then
	\begin{equation}
	\sum_{k=0}^n (-1)^{n-k} \binom{n-b}{k-b} H_{k-b} F_{2k} = \sum_{k=0}^n \binom{b}{n-k} \left(H_b + H_{n-b} - H_{b-n+k}\right)F_k.
	\end{equation}
	More generally,
	\begin{equation}\label{today55}
	\sum_{k=0}^n (-1)^{n-k} \binom{n-b}{k-b} H_{k-b} G_{2k+t} = \sum_{k=0}^n \binom{b}{n-k} \left(H_b + H_{n-b} - H_{b-n+k}\right)G_{k+t}.
	\end{equation}
	\end{theorem}
	
	\begin{theorem}
	For $n\in \mathbb{N}$ and $b\in \mathbb{C}\setminus\mathbb{Z}^{-}$ and $x$ is a complex variable, then
	\begin{equation}\label{today56}
	\sum_{k=0}^n (-1)^{n-k} \binom{n-b}{k-b} H_{k-b} G_{k+t} = \sum_{k=0}^n (-1)^k \binom{b}{n-k} \left(H_b + H_{n-b} - H_{b-n+k}\right)G_{k+t}.
	\end{equation}
	\end{theorem}
	\begin{theorem}
	For $n\in \mathbb{N}$ and $b\in \mathbb{C}\setminus\mathbb{Z}^{-}$ and $x$ is a complex variable, then
	\begin{equation}\label{today57}
	\sum_{k=0}^n (-1)^{n-k} 2^k \binom{n-b}{k-b} H_{k-b} G_{2k+t} = \sum_{k=0}^n \binom{b}{n-k} \left(H_b+H_{n-b}-H_{b-n+k}\right)G_{3k+t}
	\end{equation}
	and
	\begin{equation}\label{today58}
	\sum_{k=0}^n (-1)^{n-k+1} 2^k \binom{n-b}{k-b} H_{k-b} G_{k+t} = \sum_{k=0}^n (-1)^k \binom{b}{n-k}\left(H_b+H_{n-b}-H_{b-n+k}\right)G_{3k+t}.
	\end{equation}
	\end{theorem}
	\begin{proof}
	Set $x=\pm\alpha$, $x=\pm\alpha^3$ in identity~\eqref{eq.gldge7l}, the results above follow immediately.
	\end{proof}

	\section{Method III: Fibonacci-harmonic number sums from binomial transformation identities}

	Let $\{(s_n),(\sigma_n)\}$, $n=0,1,2,\dots$, be a binomial-transform pair; that is let
	\begin{equation*}
	\sigma_n = \sum_{k = 0}^n ( - 1)^k \binom{{n}}{k} s_k, \qquad s_n = \sum_{k = 0}^n ( - 1)^k \binom{{n}}{k}\sigma_k.
	\end{equation*}
	
	\subsection{Sums from binomial transformation identities of Boyadzhiev}
	
	Boyadzhiev~\cite{boyadzhiev16} derived the results stated in Lemmata~\ref{lem.boyad1} and~\ref{lem.boyad2}.
	
	\begin{lemma}\label{lem.boyad1}
		If $n$ is a non-negative integer and $\{(s_k),(\sigma_k)\}$, $k=0,1,2,\dots$, is a binomial-transform pair, then
		\begin{equation}\label{boyad1}
			\sum_{k = 0}^n {( - 1)^k \binom{{n}}{k}H_k s _k } = H_n \sigma_n - \sum_{k = 0}^{n - 1} {\frac{{\sigma_k }}{{n - k}}}.
		\end{equation}
	\end{lemma}
	
	\begin{lemma}
	If $n$ is a non-negative integer and $r$ and $t$ are integers, then
	\begin{equation}\label{eq.oxjlkzh}
	\sum_{k = 0}^n (- 1)^k \binom{{n}}{k}\frac{G_{tk + r}}{L_t^k} = \frac{( - 1)^r}{L_t^n}\left( G_0 L_{tn - r} - G_{tn - r} \right).
	\end{equation}
	\end{lemma}
	\begin{proof}
	Set $x=\alpha^t/L_t$ and $x=\beta^t/L_t$, in turn, in the binomial theorem
	\begin{equation*}
	\sum_{k = 0}^n ( - 1)^k \binom{{n}}{k}x^k = \left( {1 - x} \right)^n.
	\end{equation*}
	Thus, we get
	\begin{align*}
	&\sum_{k=0}^n (-1)^k \binom{n}{k}\frac{G_{tk+r}}{L_{t}^k} = \frac{\alpha^r(G_1-G_0\beta)}{(\alpha-\beta)L^n_t}\left(L_t-\alpha^t\right)^n+\frac{\beta^r(\alpha G_0-G_1)}{(\alpha-\beta)L^n_t}\left(L_t-\beta^t\right)^n \\
	&=\frac{(-1)^r}{(\alpha-\beta)L_t^n}\left(\beta^{tn-r}\left(G_1-G_0\beta-\alpha G_0+\alpha G_0\right)+\alpha^{tn-r}\left(\alpha G_0-G_0\beta+G_0\beta-b\right)\right) \\
	&=\frac{(-1)^r}{L_t^n}\left(G_0(\alpha^{tn-r}+\beta^{tn-r})-\frac{(G_1-G_0\beta)\alpha^{tn-r}+(\alpha G_0-G_1)\beta^{tn-r}}{\alpha-\beta}\right) \\
	&=\frac{{( - 1)^r }}{{L_t^n }}\left( {G_0 L_{tn - r} - G_{tn - r} } \right).
	\end{align*}
	\end{proof}
	
	\begin{theorem}
		If $n$ is a non-negative integer and $r$ and $t$ are integers, then
		\begin{align}
			&\sum_{k = 0}^n {( - 1)^k \binom{{n}}{k}L_t^{n - k} H_k G_{tk + r} }\nonumber\\
			&\qquad  = ( - 1)^r H_n \left( {G_0 L_{tn - r}  - G_{tn - r} } \right) - ( - 1)^r \sum_{k = 0}^{n - 1} {\frac{{L_t^{n - k} }}{{n - k}}\left( {G_0 L_{tk - r}  - G_{tk - r} } \right)} .
		\end{align}
	\end{theorem}
	\begin{proof}
		With~\eqref{eq.oxjlkzh} in mind, choose
		\begin{equation*}
			s_k  = \frac{{G_{tk + r} }}{{L_t^k }},\quad\sigma _k  = \frac{{( - 1)^r }}{{L_t^k }}\left(G_0L_{tk-r}-G_{tk-r}\right),
		\end{equation*}
		and use these in~\eqref{boyad1}.
	\end{proof}
	
\begin{corollary}
If $n$ is a non-negative integer and $r$ and $t$ are integers, then
\begin{align}
&\sum_{k = 0}^n ( - 1)^k \binom{{n}}{k}L_t^{n - k} H_k F_{tk + r} = ( - 1)^{r + 1} H_n F_{tn - r} + ( - 1)^r \sum_{k = 0}^{n - 1} {\frac{{L_t^{n - k} F_{tk - r} }}{{n - k}}} ,\\
&\sum_{k = 0}^n ( - 1)^k \binom{{n}}{k}L_t^{n - k} H_k L_{tk + r} = ( - 1)^r H_n L_{tn - r} - ( - 1)^r \sum_{k = 0}^{n - 1} {\frac{{L_t^{n - k} L_{tk - r} }}{{n - k}}}.
\end{align}
In particular
\begin{align}
&\sum_{k = 0}^n ( - 1)^k \binom{{n}}{k}H_k F_k = - H_n F_n + \sum_{k = 0}^{n - 1} {\frac{{F_k }}{{n - k}}}\label{eq.lcrsozt},\\ 
&\sum_{k = 0}^n ( - 1)^k \binom{{n}}{k}H_k L_k = H_n L_n - \sum_{k = 0}^{n - 1} {\frac{{L_k }}{{n - k}}} .
	\end{align}
	\end{corollary}

\begin{remark}
Boyadzhiev~\cite{boyadzhiev16} also obtained~\eqref{eq.lcrsozt}.
\end{remark}

\begin{theorem}
If $n$ is a non-negative integer and $r$ and $t$ are integers, then
\begin{align}
\begin{split}
\sum_{k=0}^n (-1)^{k+r} \binom{n}{k} H_k L_t^{n-k}\left(G_0L_{tk-r}-G_{tk-r}\right)
= H_n G_{tn+r} - \sum_{k=0}^{n-1} \frac{L_t^{n-k}G_{tk+r}}{n-k}.
\end{split}
\end{align}
\end{theorem}	
\begin{proof}
Choose
\begin{equation*}
s_k  =\frac{{( - 1)^r }}{{L_t^k }}\left(G_0L_{tk-r}-G_{tk-r}\right) ,\quad\sigma_k  = \frac{{G_{tk + r} }}{{L_t^k }},
\end{equation*}
and use these in~\eqref{boyad1}.
\end{proof}	

\begin{corollary}
If $n$ is a non-negative integer and $r$ and $t$ are integers, then
\begin{align}
&\sum_{k=0}^n (-1)^{k+r+1} \binom{n}{k} L_t^{n-k} H_k F_{tk-r} = H_n F_{tn+r} - \sum_{k=0}^{n-1} \frac{L_t^{n-k}F_{tk+r}}{n-k}, \\
&\sum_{k=0}^n (-1)^{k+r} \binom{n}{k} L_t^{n-k} H_k L_{tk-r} = H_n L_{tn+r} - \sum_{k=0}^{n-1} \frac{L_t^{n-k}L_{tk+r}}{n-k}.
\end{align}
\end{corollary}

\begin{lemma}\label{lem.boyad2}
If $n$ is a non-negative integer and $\{(s_k),(\sigma_k)\}$, $k=0,1,2,\dots$, is a binomial-transform pair, then
\begin{equation}\label{boyad2}
\sum_{k = 0}^n {( - 1)^{k + 1} \binom{{n}}{k}G_k s _k }  = \sum_{k = 0}^n {( - 1)^k \binom{{n}}{k}G_{n - 2k} \sigma_k } .
\end{equation}
\end{lemma}
	
\begin{theorem}
If $n$ is a non-negative integer and $m$ is a complex number that is not a negative integer, then
\begin{equation}\label{eq.uwn5wo9}
\sum_{k = 0}^n {( - 1)^{k + 1} \binom{{n}}{k} \frac{{m\,G_k H_{k + m} }}{{k + m}}} 
= \sum_{k = 0}^n {( - 1)^k \binom{{n}}{k}\binom{{k + m}}{m}^{ - 1} G_{n - 2k} \left( {H_{k + m}  - H_k } \right)},\\
\end{equation}
and
\begin{equation}\label{eq.mlvhmnt}
\sum_{k = 0}^n {( - 1)^{k + 1} \binom{{n}}{k}\binom{{k + m}}{m}^{ - 1} G_k \left( {H_{k + m} - H_k } \right)} 
= \sum_{k = 0}^n {( - 1)^k \binom{{n}}{k}\frac{m}{{k + m}}G_{n - 2k} H_{k + m} }. 
\end{equation}
In particular,
\begin{equation}
\sum_{k = 0}^n {( - 1)^{k + 1} \binom{{n}}{k}\frac{{G_k H_{k + 1} }}{{k + 1}}}  
= \sum_{k = 0}^n {( - 1)^k \binom{{n}}{k}\frac{{G_{n - 2k} }}{{\left( {k + 1} \right)^2 }}}, 
\end{equation}
and
\begin{equation}
\sum_{k = 0}^n {( - 1)^{k + 1} \binom{{n}}{k}\frac{{G_k }}{{\left( {k + 1} \right)^2 }}}  
= \sum_{k = 0}^n {( - 1)^k \binom{{n}}{k}\frac{{G_{n - 2k} H_{k + 1} }}{{k + 1}}}.
\end{equation}
\end{theorem}
\begin{proof}
Consider the following identity (Boyadzhiev~\cite[Identity (9.46)]{boyadzhiev18}):
\begin{equation*}
\sum_{k = 0}^n {( - 1)^k \binom{{n}}{k}\frac{{H_{k + m} }}{{k + m}}} = \frac{{H_{n + m} - H_n }}{m}\binom{{n + m}}{n}^{- 1},\quad m\ne 0,
\end{equation*}
and use
\begin{equation*}
s_k = \frac{{H_{k + m} }}{{k + m}},\quad\sigma_k = \frac{{H_{k + m} - H_k }}{m}\binom{{k + m}}{k}^{- 1}, 
\end{equation*}
in~\eqref{boyad2} to obtain~\eqref{eq.uwn5wo9}. Identity~\eqref{eq.mlvhmnt} is obtained from~\eqref{eq.uwn5wo9} by symmetry.
\end{proof}

\begin{prop}
We have the following identities;
\begin{align}
\sum_{k=0}^n (-1)^{k+1} \binom{n}{k}\frac{G_k}{2k-1} &= \sum_{k=0}^n (-1)^{k+1} 2^{2k}\frac{\binom{n}{k}}{\binom{2k}{k}}G_{n-2k}, \label{irration1}\\
\sum_{k=0}^n (-1)^{k+1} \binom{n}{k}\frac{O_kG_k}{2k-1} &= \sum_{k=0}^n (-1)^{k} 2^{2k-1}\frac{\binom{n}{k}}{\binom{2k}{k}}\left(H_k-2O_k\right)G_{n-2k}.\label{irration2}
\end{align}	
\end{prop}
\begin{proof}
Setting $m=-1/2$ in each of~\eqref{eq.uwn5wo9} and~\eqref{eq.mlvhmnt} yield two new identities :
\begin{align}
&2\sum_{k=0}^n (-1)^{k+1} \binom{n}{k} \frac{\ln 2-O_k}{2k-1}G_k = \sum_{k=0}^n (-1)^k 2^{2k}\frac{\binom{n}{k}}{\binom{2k}{k}}(2O_k-H_k-2\ln 2)G_{n-2k}, \label{new1}\\
&\sum_{k=0}^n (-1)^{k+1} 2^{2k} \frac{\binom{n}{k}}{\binom{2k}{k}}(2O_k-H_k-2\ln 2)G_{k} = 2\sum_{k=0}^n (-1)^{k+1} \binom{n}{k}\frac{O_k-\ln 2}{2k-1}G_{n-2k}. \label{new2}
\end{align}
The presence of the natural logarithm in equation \eqref{new1} tells us that the equation subsumes two separate identities. We equate the coefficients of $\ln 2$ and also equate the rational parts from both sides and the results follow.
\end{proof}

\begin{prop}
We have
\begin{align}
\sum_{k=0}^n (-1)^k \binom{n}{k}\frac{G_{n-2k}}{2k-1} &= \sum_{k=0}^n (-1)^k 2^{2k}\frac{\binom{n}{k}}{\binom{2k}{k}}G_k, \label{irration3}\\
\sum_{k=0}^n (-1)^{k+1} \binom{n}{k}\frac{O_k G_{n-2k}}{2k-1} &= \sum_{k=0}^n (-1)^{k+1}2^{2k-1}\frac{\binom{n}{k}}{\binom{2k}{k}}\left(2O_k-H_k\right)G_k.\label{irration4}
\end{align}
\end{prop}
\begin{proof}
Apply a similar argument as in Proposition 1.
\end{proof}

\begin{theorem}
If $n$ is a non-negative integer and $m$ is a complex number that is not a negative integer, then
\begin{equation}\label{eq.x63o4jh}
\sum_{k = 1}^n {( - 1)^{k + 1} \binom{{n}}{k}G_k H_{k + m} } = G_n H_m + \sum_{k = 1}^n {( - 1)^{k + 1} \binom{{n}}{k}\binom{{k + m}}{m}^{-1}\frac{{G_{n - 2k} }}{k}} 
\end{equation}
and
\begin{equation}
\sum_{k = 1}^n (- 1)^k \binom{{n}}{k}\binom{{k + m}}{m}^{ - 1} \frac{{G_k}}{k} = \sum_{k = 0}^n ( - 1)^k \binom{{n}}{k}G_{n - 2k} H_{k + m}.
\end{equation}
In particular,
\begin{equation}
\sum_{k = 1}^n ( - 1)^k \binom{{n}}{k} G_k H_k = \sum_{k = 1}^n {( - 1)^k \binom{{n}}{k}\frac{{G_{n - 2k} }}{k}} 
\end{equation}
and
\begin{equation}
\sum_{k = 1}^n ( - 1)^{k + 1} \binom{{n}}{k} \frac{{G_k }}{k} = \sum_{k = 1}^n {( - 1)^k \binom{{n}}{k}G_{n - 2k} H_k }.
\end{equation}
\end{theorem}
\begin{proof}
Consider the following identity (Boyadzhiev~\cite[Identity (9.37)]{boyadzhiev18}):
\begin{equation*}
\sum_{k = 0}^n {( - 1)^k \binom{{n}}{k}H_{k + m} } = - \frac{1}{n}\binom{{n + m}}{m}^{ - 1},\quad n\ne0,
\end{equation*}
and use
\begin{equation*}
s_k = H_{k + m} ,\quad \sigma_k = -\frac{{1 }}{{k }}\binom{{k + m}}{m}^{ - 1},  
\end{equation*}
in~\eqref{boyad2} to obtain~\eqref{eq.x63o4jh}.
\end{proof}
	
\subsection{Sums from binomial transformation identities of Gould and Quaintance}
	
Gould and Quaintance~\cite{gould14} proved a binomial transform identity that is equivalent to the following:
\begin{lemma}
If $n$ is a non-negative integer, $m$ and $r$ are complex numbers other than negative integers and $\{(t_k),(\tau_k)\}$, $k=0,1,2,\dots$, 
is a binomial-transform pair, then
\begin{align}\label{gouldqu}
\sum_{k = 0}^n {( - 1)^k \binom{{n}}{k}\binom{{r + m + n - k + 1}}{{m + 1}}^{-1}t_k } = \frac{{m + 1}}{{r + 1}}\sum_{k = 0}^n {( - 1)^{n - k} \binom{{n}}{k}\binom{{r + m + n - k + 1}}{{r + 1}}^{-1}\tau_k } .
\end{align}
\end{lemma}
	
\begin{theorem}\label{gooday}
If $n$ is a non-negative integer, $s$ and $t$ are integers and $r$ and $m$ are complex numbers other than negative integers, then 
\begin{align}\label{gooday1}
\begin{split}
&\sum_{k=0}^n (-1)^k \binom{n}{k} \binom{r+m+n-k+1}{m+1}^{-1}\left(H_{m+1}-H_{r+m+n-k+1}\right)\frac{G_{tk+s}}{L_{t}^k} \\
&= \frac{m+1}{r+1}\sum_{k=0}^n (-1)^{n-k}\binom{n}{k}\binom{r+m+n-k+1}{r+1}^{-1}\left(H_{m+n-k}-H_{r+m+n-k+1}\right)\frac{(-1)^s}{L^k_t}\left(G_0L_{tk-s}-G_{tk-s}\right) \\
&\qquad +\frac1{r+1}\sum_{k=0}^n (-1)^{n-k} \binom{n}{k}\binom{r+m+n-k+1}{r+1}^{-1}\frac{(-1)^s}{L^k_t}\left(G_0L_{tk-s}-G_{tk-s}\right).
\end{split}
\end{align}
\end{theorem}
\begin{proof}
Differentiate~\eqref{gouldqu} with respect to $m$, choose a binomial-transform pair
$$t_k = \frac{G_{tk+s}}{L^k_t}, \quad \tau_k=\frac{(-1)^s}{L^k_t}\left(G_0L_{tk-s}-G_{tk-s}\right).$$
Note that
\begin{equation*}
\frac{d}{{dm}}\binom{{r + m + n - k + 1}}{{m + 1}}^{ - 1} = \binom{{r + m + n - k + 1}}{{m + 1}}^{ - 1} \left( {H_{m + 1} - H_{r + m + n - k + 1} } \right),
\end{equation*}
and
\begin{equation*}
\frac{d}{{dm}}\binom{{r + m + n - k + 1}}{{r + 1}}^{ - 1} = \binom{{r + m + n - k + 1}}{{r + 1}}^{ - 1} \left( {H_{m + n - k} - H_{r + m + n - k + 1} } \right).
\end{equation*}
\end{proof}

\begin{corollary}
If $n$ is a non-negative integer, $s$ and $t$ are integers and $r$ is complex number other than negative integers, then 
\begin{align}\label{go13}
\begin{split}
&\sum_{k=0}^n(-1)^k \binom{n}{k}\frac{\left(1-H_{r+n-k+1}\right)}{(r+n-k+1)}\frac{G_{tk+s}}{L_t^k} \\
&= \frac{1}{r+1}\sum_{k=0}^n (-1)^{n-k-s}\binom{n}{k}\binom{r+n-k+1}{r+1}^{-1}\left(H_{n-k}-H_{r+n-k+1}\right)\frac{\left(G_0L_{tk-s}-G_{tk-s}\right)}{L_t^k}\\
&\quad + \frac{1}{r+1}\sum_{k=0}^n (-1)^{n-k-s}\binom{n}{k}\binom{r+n-k+1}{r+1}^{-1}\frac{\left(G_0L_{tk-s}-G_{tk-s}\right)}{L_t^k}.	
\end{split}	
\end{align}
In particular, if $s$ and $t$ are integers, then
\begin{align}\label{may1}
\begin{split}
\sum_{k=0}^n (-1)^k\binom{n}{k}\frac{(1-H_{n-k+1})}{(n-k+1)}\frac{G_{tk+s}}{L_t^k}
&= \sum_{k = 0}^n(-1)^{n-k-s}\binom{n}{k}\frac{\left(G_0L_{tk+s}-G_{tk-s}\right)}{(n-k+1)L_t^k}\\
&\qquad - \sum_{k = 0}^n(-1)^{n-k-s}\binom{n}{k}\frac{\left(G_0L_{tk+s}-G_{tk-s}\right)}{(n-k+1)^2L_t^k}.
\end{split}
\end{align}
\end{corollary}
\begin{proof}
Set $m=0$ in \eqref{gooday1}, then the result follows.
\end{proof}

From the LHS of~\eqref{gooday1} with $m=-3/2$, we have
\begin{equation}\label{mn5v94x}
\binom{{r + n - k - 1/2}}{{ - 1/2}}^{ - 1} = \binom{{r + n - k - 1/2}}{{r + n - k}}^{ - 1} = \binom{{2\left( {r + n - k} \right)}}{{r + n - k}}^{ - 1} 2^{2\left( {r + n - k} \right)} ,\text{ by~\eqref{eq.qdgxupm}.}
\end{equation}
Also
\begin{equation}\label{mdxkxi4}
H_{ - 1/2}  - H_{ - 1/2 + r + n - k} = - 2O_{r + n - k},\text{ by~\eqref{eq.plh634k}}.
\end{equation}
Similarly,
\begin{align}\label{zj8qis2}
H_{ - 3/2 + n - k} - H_{ - 1/2 + r + n - k} &= H_{ - 1/2 + n - k - 1}  - H_{ - 1/2 + r + n - k}\nonumber \\
& = \left( {2O_{n - k - 1} - 2\ln 2} \right) - \left( {2O_{r + n - k}  - 2\ln 2} \right)\nonumber\\
& = 2\left( {O_{n - k - 1} - O_{r + n - k} } \right)\text{ by~\eqref{eq.zts4fm1}}.
\end{align}
Using~\eqref{eq.qdgxupm}, for $k<n$ we have;
\begin{equation}\label{kgwmxux}
\binom{{r + n - k - 1/2}}{{r + 1}}^{ - 1} = \binom{{2\left( {r + n - k} \right)}}{{r + 1}}^{ - 1} \binom{{r + n - k}}{{r + 1}}^{ - 1} \binom{{2\left( {n - k - 1} \right)}}{{n - k - 1}}2^{2(r + 1)}. 
\end{equation}
Also for $k=n$, we use;
\begin{equation}\label{kgwmxuxuu}
	\binom{{r  - 1/2}}{{r + 1}}^{ - 1}  = -\frac{2^{-2(r + 1)}}{(r+1)}\binom{2r}{r}. 
\end{equation}

Now, we state a result involving odd Harmonic numbers.
\begin{corollary}
If $n$ is a non-negative integer, $s$ and $t$ are integers and $r$ is a complex number other than negative integers, then 
\begin{align}\label{go14}
\begin{split}
&\sum_{k=0}^n(-1)^{k+1}\binom{n}{k}\binom{2(r+n-k)}{r+n-k}^{-1}2^{2(r+n-k)+1}\frac{O_{r+n-k}G_{tk+s}}{L_t^k} \\
&= \frac{2^{2(r+1)}}{r+1}\sum_{k=0}^{n-1}(-1)^{n-k+1+s}\frac{\binom{n}{k}\binom{2(n-k-1)}{n-k-1}}{\binom{2(r+n-k)}{r+1}\binom{r+n-k}{r+1}}\frac{\left(O_{n-k-1}-O_{r+n-k}\right)\left(G_0L_{tk-s}-G_{tk-s}\right)}{L_t^k}\\
&\quad +\frac{2^{2(r+1)}}{r+1}\sum_{k=0}^{n-1}(-1)^{n-k+1+s}\frac{\binom{n}{k}\binom{2(n-k-1)}{n-k-1}}{\binom{2(r+n-k)}{r+1}\binom{r+n-k}{r+1}}\frac{\left(G_0L_{tk-s}-G_{tk-s}\right)}{L_t^k} \\
&\quad +\frac{(-1)^s}{(r+1)^2}\binom{2r}{r}\frac{(1-O_r)\left(G_0L_{tn-s}-G_{tn-s}\right)}{L_t^n}.
\end{split}
\end{align}
\end{corollary}
\begin{proof}
Set $m=-3/2$ in \eqref{gooday1}.
\end{proof}
In particular, if $s$ and $t$ are integers, then
\begin{align}\label{may10}
\begin{split}
&\sum_{k=0}^n (-1)^{k+1}2^{2(n-k)}\frac{\binom{n}{k}}{\binom{2n-2k}{n-k}}\frac{O_{n-k}G_{tk+s}}{L_t^k}\\
&=\sum_{k=0}^{n-1}(-1)^{n-k+s+1}\frac{\binom{n}{k}\binom{2(n-k-1)}{n-k-1}}{(n-k)^2}\frac{\left(G_0L_{tk-s}-G_{tk-s}\right)}{L_t^k} \\
&-\sum_{k=0}^{n-1}(-1)^{n-k+s+1}\frac{\binom{n}{k}\binom{2(n-k-1)}{n-k-1}}{(n-k)^2(2n-2k-1)}\frac{\left(G_0L_{tk-s}-G_{tk-s}\right)}{L_t^k}.
\end{split}
\end{align}

\begin{theorem}\label{gooday200}
	If $n$ is a non-negative integer, $s$ and $t$ are integers and $r$ and $m$ are complex numbers other than negative integers, then 
	\begin{align}\label{gooday100}
		\begin{split}
			&\sum_{k=0}^n(-1)^k\binom{n}{k}\binom{r+m+n-k+1}{m+1}^{-1}\left(H_{m+1}-H_{r+m+n-k+1}\right)\frac{(-1)^s}{L^k_t}\left(G_0L_{tk-s}-G_{tk-s}\right)\\
			&=\frac{m+1}{r+1}\sum_{k=0}^n(-1)^{n-k}\binom{n}{k}\binom{r+m+n-k+1}{r+1}^{-1}\left(H_{m+n-k}-H_{r+m+n-k+1}\right)\frac{G_{tk+s}}{L_{t}^k}\\
			&\qquad +\frac1{r+1}\sum_{k=0}^n(-1)^{n-k}\binom{n}{k}\binom{r+m+n-k+1}{r+1}^{-1}\frac{G_{tk+s}}{L_{t}^k}.
		\end{split}
	\end{align}
\end{theorem}
\begin{proof}

		Differentiate~\eqref{gouldqu} with respect to $m$, choose a binomial-transform pair
		$$t_k=\frac{(-1)^s}{L^k_t}\left(G_0L_{tk-s}-G_{tk-s}\right), \quad \tau_k=\frac{G_{tk+s}}{L^k_t}.$$
\end{proof}
\begin{corollary}
If $n$ is a non-negative integer, $s$ and $t$ are integers and $r$ is a complex number other than negative integers, then 
\begin{align}\label{go130}
	\begin{split}
		&\sum_{k=0}^n(-1)^{k-s}\binom{n}{k}\frac{\left(1-H_{r+n-k+1}\right)}{(r+n-k+1)}\frac{\left(G_0L_{tk-s}-G_{tk-s}\right)}{L_t^k}\\
		&=\frac{1}{r+1}\sum_{k=0}^n(-1)^{n-k}\binom{n}{k}\binom{r+n-k+1}{r+1}^{-1}\left(H_{n-k}-H_{r+n-k+1}\right)\frac{G_{tk+s}}{L_t^k}\\
		&\quad +\frac{1}{r+1}\sum_{k=0}^n(-1)^{n-k}\binom{n}{k}\binom{r+n-k+1}{r+1}^{-1}\frac{G_{tk+s}}{L_t^k}.	
	\end{split}	
\end{align}
\end{corollary}
\begin{proof}
Set $m=0$ in Theorem~\ref{gooday200}. Thus, the result follows immediately.
\end{proof}
In particular, if $s$ and $t$ are integers, then
\begin{align}
	\label{may117}
	\begin{split}
		\sum_{k=0}^n (-1)^k\binom{n}{k}\frac{(1-H_{n-k+1})}{(n-k+1)}\frac{\left(G_0L_{tk+s}-G_{tk-s}\right)}{L_t^k}
		&=\sum_{k = 0}^n(-1)^{n-k-s}\binom{n}{k}\frac{G_{tk+s}}{(n-k+1)L_t^k}\\
		&\quad -\sum_{k = 0}^n(-1)^{n-k-s}\binom{n}{k}\frac{G_{tk+s}}{(n-k+1)^2L_t^k}.
	\end{split}
\end{align}

\begin{corollary}
If $n$ is a non-negative integer, $s$ and $t$ are integers and $r$ is a complex number other than negative integers, then 
\begin{align}\label{go140}
	\begin{split}
		&\sum_{k=0}^n(-1)^{k+1-s}\binom{n}{k}\binom{2(r+n-k)}{r+n-k}^{-1}2^{2(r+n-k)+1}\frac{O_{r+n-k}\left(G_0L_{tk-s}-G_{tk-s}\right)}{L_t^k}\\
		&=\frac{2^{2(r+1)}}{r+1}\sum_{k=0}^{n-1}(-1)^{n-k+1}\frac{\binom{n}{k}\binom{2(n-k-1)}{n-k-1}}{\binom{2(r+n-k)}{r+1}\binom{r+n-k}{r+1}}\frac{\left(O_{n-k-1}-O_{r+n-k}\right)G_{tk+s}}{L_t^k}\\
		&\quad +\frac{2^{2(r+1)}}{r+1}\sum_{k=0}^{n-1}(-1)^{n-k+1}\frac{\binom{n}{k}\binom{2(n-k-1)}{n-k-1}}{\binom{2(r+n-k)}{r+1}\binom{r+n-k}{r+1}}\frac{G_{tk+s}}{L_t^k}+\frac{(-1)^s}{(r+1)^2}\binom{2r}{r}\frac{(1-O_r)G_{tn+s}}{L_t^n}.
	\end{split}
\end{align}
\end{corollary}
\begin{proof}
Set $m=-3/2$ in Theorem~\ref{gooday200}.
\end{proof}
In particular, if $s$ and $t$ are integers, then
\begin{align}\label{may10}
	\begin{split}
		&\sum_{k=0}^n (-1)^{k+1}2^{2(n-k)}\frac{\binom{n}{k}}{\binom{2n-2k}{n-k}}\frac{O_{n-k}\left(G_0L_{tk-s}-G_{tk-s}\right)}{L_t^k}\\
		&=\sum_{k=0}^{n-1}(-1)^{n-k+s+1}\frac{\binom{n}{k}\binom{2(n-k-1)}{n-k-1}}{(n-k)^2}\frac{G_{tk+s}}{L_t^k}
		 -\sum_{k=0}^{n-1}(-1)^{n-k+s+1}\frac{\binom{n}{k}\binom{2(n-k-1)}{n-k-1}}{(n-k)^2(2n-2k-1)}\frac{G_{tk+s}}{L_t^k}.
	\end{split}
\end{align}

	

\end{document}